\documentclass[IEEEtran,11pt]{article}
\usepackage{latexsym,mathtools}
\usepackage{graphicx,amsmath,amsfonts,amssymb,color,mathrsfs,placeins,varioref}

\usepackage{epsfig}

 \newcommand{\QED}{\hfill \thicklines \framebox(6.6,6.6)[l]{}}
 \newenvironment{proof}{\noindent {\bf \sc Proof.} \rm}{\QED}

\numberwithin{equation}{section}

 \newtheorem{theorem}{Theorem}[section]
 \newtheorem{lemma}{Lemma}[section]
 \newtheorem{proposition}{Proposition}[section]

 \newtheorem{definition}{Definition}[section]

\newcommand{\eqnb}{\begin{eqnarray*}}
\newcommand{\eqne}{\end{eqnarray*}}

\def\beqlb{\begin{eqnarray}}\def\eeqlb{\end{eqnarray}}
\def\beqnn{\begin{eqnarray*}}\def\eeqnn{\end{eqnarray*}}

\title{\Large \bf Construction of New Copulas with Queueing Application}
\author{
Suman Thapa   \\
School of Mathematics and Statistics \\
\vspace*{4mm}
Carleton University, Ottawa, ON Canada K1S 5B6 \\
Yiqiang Q. Zhao\\
School of Mathematics and Statistics \\
Carleton University, Ottawa, ON Canada K1S 5B6}
\date{January, 2021}

\begin{document}
\maketitle
\begin{abstract}
In this paper, we construct a bound copula, which can reach both Frechet's lower and upper bounds for perfect positive and negative dependence cases. Since it covers a wide range of dependency and simple for computational purposes, it can be very useful. We then develop a new perturbed copula using the lower and upper bounds of Frechet copula and show that it satisfies all properties of a copula. In some cases, it is very difficult to get results such as distribution functions and the expected values in explicit form by using copulas such as Archemedes, Guassian, $t$-copula. Thus, we can use these new copulas. For both copulas, we derive the strength of measures of the dependency such as Spearman's rho, Kendall's tau, Blomqvist's beta and Gini's gamma, and the coefficients of the tail dependency. As an application, we use the bound copula to analyze the dependency between two service times to evaluate the mean waiting time and the mean service time when customers launch two replicas of each task on two parallel servers using the cancel-on-finish policy. We assume that the inter-arrival time is exponential and the service time is general.
\vspace*{5mm}\\

\medskip

\end{abstract}

\section{Introduction}
There are several methods for constructing copulas in the literature, for example, see  \cite{nelsen2007introduction}, \cite{joe2014dependence}, and \cite{durante2010copula}. Moreover, in their books \cite{nelsen2003properties} and \cite{jaworski2010copula}, the authors studied copulas, constructions and applications. Examples of methods of constructions include the inversion method, geometric method, algebraic method among others. \\
After a literature review, we realized that it is of interest to construct copulas which are simple in nature such that we can find the distribution function and the expected value of joint behavior of two dependent components in explicit form. Constructing new copulas are also necessary for varieties of applications with data of different natures.

We know that some copulas such as FGM, Ali-Mikhael-Haq, Gaussian copula have weak dependency in both (upper and lower) tails. Frank copula allows a wide range of dependency with weak dependency in both tails.  Clayton copula has a strong lower tail dependency and a weak upper tail dependency whereas Gumbel copula has a weak lower tail dependency and a strong upper tail dependency. Thus, in the existing copulas, we cannot find a copula which satisfies Frechet bounds for perfect positive and negative dependence case and a strong dependence in both tails. The bound copula constructed in this paper satisfies all required properties. It is suitable to analyze strong dependence in the tails. Since it covers a wide range of dependency, the bound copula can be applied in all types of dependency. We also derive the measures of dependency such as Spearman's rho, Kendall's tau, Blomqvist's beta and Gini's gamma and the tail dependency for the both copulas. Charpentier \cite{charpentier2003tail} derived the tail distribution and dependence measures.

Next, we present another new perturbation of bivariate copula in terms of its lower and upper Frechet bounds. In this area, Komornik, Komornikova and Kalicka published  a paper, \cite{komornik2017dependence}, studying perturbations of copulas, and  Fernandez-Sanchez and Ubeda-Flores published a paper, \cite{fernandez2019solution}, for perturbations of the product copula. Our construction is to use Frechet lower and upper bounds, which, we believe, is new and useful. Sometimes a minor perturbation of copula fits to real data better than the original copula. So, we need to change the copula accordingly by using its functions or relations. This idea has motivated us to construct this new perturbed copula and investigate measures of dependency and tail dependencies. The difference between the Frechet bounds of any copula is not a copula, but we prove that if we add the difference of Frechet bounds of any copula to this copula, it will satisfy all properties of a copula.

For application of the new constructed bound copula, we analyze the dependency between two parallel service times when two replicas of each task, we launched are on two parallel dependent servers and evaluate the mean waiting time and the mean service time. In many situations, when one server is faster (or slower), it motivates the other server to be faster (or slower) as well. This is a positive dependence between the servers. In some other situations, if one server is working at a relaxed pace, the other server has to work faster to balance the load, which is an example of a negative dependence between the servers. As the bound copula  is simple to use and feasible to analyze all types (weak to strong) of dependency, it is a favourable choice to apply. We assume that the inter-arrival time of tasks is exponential and two different service times for a task, namely, the shifted exponential, and the hypo-exponential distribution are considered.

In queueing applications, dependency between two components often exist. Patil and Naik-Nimbalkar \cite{patil2014conditional} and Turner \cite{turner1979conditional} considered the dependency between the waiting time and the queue length. Behazad and Rad \cite{behzad2018simultaneous} discussed the dependence between two waiting times when customers arrive to two different queues simultaneously. Muller \cite{muller2000waiting} discussed the dependency between inter-arrival and service times. In other related papers, Lee and Wang \cite{lee2011waiting} derived waiting time probabilities in the $M/G/1+ M$ queue. Raaijmakers, Albrecher, and Boxma \cite{raaijmakers2019single} stidied a single server queue with mixing dependencies. Kumar \cite{kumar2010probability} derived probability distributions and estimations of Ali-Mikhail-Haq copula.

The rest of the paper is organized as  follows: In Section 2, we construct two new copulas, the bound copula and a perturbed copula, and investigate measures of dependency and the tail dependencies. In Section 3, the dependence between two service times is analyzed which includes the waiting time and service time analysis in Section 3.1 and the dependence between two service times using the bound copula in Section 3.2 and Conclusions are made in Section 4.

\section{Construction of new copulas}
In this section, we present two new copulas for the literature, which are the bound copula and a perturbed copula, and derive their measures of dependency and the tail dependencies. In some cases, the existing copulas such as Archemedes (Gumbel, Clayton, Frank, Ali-Mikheil-Haq), Gaussian and $t$ copulas are not simple in computational point of view. Thus, it is comparatively easier to find the distribution functions and the expected values of joint and conditional behavior of two components in explicit form using these new copulas.

\subsection{Bound copula}
For any bivariate copula $C(u,v)$, the inequality, \cite{nelsen2007introduction}, $\max(u+v-1, 0) \leq C(u,v) \leq \min(u,v)$ holds, where $\max(u+v-1,0)$ and $\min(u,v)$ are the Frechet lower and upper bounds, respectively. For a bivariate copula, both bounds are also  copulas. The lower bound $\max(u+v-1, 0)$ corresponds to the case of perfect negative dependence whereas the upper bound $\min(u,v)$ corresponds to the case of perfect positive dependence.

We realize that many copulas from the literature such as $FGM$, Ali-Mikhael-Haq, Clayton, Gumbel, Frank, Guassian, $t$ do not reach $\min(u, v)$ for perfect positive dependence and $\max(u+v-1, 0)$ for perfect negative dependence, which show strong dependence in both lower and upper tails. So, there is a need to construct a copula, which satisfies Frechet lower and upper bounds for any strength of dependency, i.e. for $\theta \in [-1,1]$. Moreover, we do not have any copula similar to this copula in the literature.

\begin{theorem} The bivariate expression $C(u,v) = (1-\theta^2)uv + \frac{\theta}{4}\Big[(1+\theta)^2 \min(u,v)- (1-\theta)^2 \max(u+v-1, 0)\Big]$ defines a copula, which is a Frechet upper bound when $\theta = 1$ and a Frechet lower bound when $\theta= -1$. Moreover, when $\theta =0$, the resulting copula corresponds to the independent case. This copula is referred to as the bound copula.
\end{theorem}
\begin{proof}
A proof is given in appendix.
\end{proof}

\subsubsection{Dependence properties of the bound copula}
In this section, we derive the strength of dependency of the copula such as Spearman's rho, Kendall's tau, Blomqvist's beta, and Gini's gamma.

\begin{theorem} The measures of dependency: Spearman's rho, Blomqvist's beta and the Gini's gamma of copula $C(u,v) = (1- \theta^2)uv +\frac{\theta}{4}\Big[(1+\theta)^2 \min(u,v)- (1-\theta)^2 \max(u+v-1,0)\Big]$ are $ \dfrac{\theta(\theta^2 + 1)}{2}$ and the Kendall's tau is $ \dfrac{\theta(8 -3\theta + 12\theta^2 + 6\theta^3 + 4\theta^4 - 3\theta^5)}{24}$.	
\end{theorem}
\begin{proof}
A proof is given in appendix.
\end{proof}

\subsubsection{Tail dependence of bound copula}
In this subsection, we derive the coefficient of the upper and the lower tail dependence for the bound copula:
$$C(u,v) = (1- \theta^2)uv + \frac{\theta}{4}\left[(1+\theta)^2 min(u,v) - (1-\theta)^2 max(u+v-1,0)\right]\cdot$$

\begin{theorem} The coefficient of the lower tail and the upper tail dependence of bound copula are both equal to \( \dfrac{\theta(1+\theta)^2}{4}\).
\end{theorem}
\begin{proof}
A proof is given in appendix.
\end{proof}

\subsection{A perturbed copula}
In this sub-section, we construct a new class of perturbation of copula in terms of its lower and upper Frechet bounds and investigate the effects of perturbation of this bivariate copula on several measures of dependency (Spearman's rho, Kendall's tau, Blomqvist's beta, Gini's gamma) and the co-efficient of the lower and the upper tail dependencies. When we use copula to a real data, to analyze the dependency, the copula must fit to the data. So, when any particular copula does not fit to real data, we can change it by using this copula accordingly to fit the data. In some cases, a minor perturbation of copula fits  the data better that the original copula. It has been demonstrated that the three measures of dependencies (Spearman's rho, Blomqvist's beta and Gini's gamma) of the perturbed copula depend on the parameter $\alpha$. Moreover, we have shown that the co-efficient of the lower and upper tail dependency is proportional to the perturbation parameter $\alpha \in [0,1]$.

\subsubsection{The difference between upper and lower bounds}
We derive the difference between the two bounds which is given in the following proposition:
\begin{proposition}
The difference between the upper bound and the lower bound is one of $u,v, 1-u, 1-v.$ That is ,
\begin{equation}
  M(u,v) - W(u,v) =
    \begin{cases}
      u, & \text{if $u \leq v$, $u + v \leq 1$}\\
      1-v, & \text{if $u \leq v$, $u + v > 1$}\\
      v,  &\text{if $u > v$, $u + v \leq 1$}\\
      1-u, & \text{if $u > v$, $u + v > 1$}\\
    \end{cases}
\end{equation}
\end{proposition}
\begin{proof}
A proof is given in the appendix.
\end{proof}

\subsubsection{Perturbation of copula}
We investigate a new copula which is obtained by the perturbation of a copula. Let $C(u,v)$ be a bivariate copula. We consider a new copula $C_{\alpha}(u,v)= C(u,v) + F_{\alpha}(u,v)$, referred to as the perturbed copula of $C(u,v)$, where $F_{\alpha}(u,v):[0,1]^2 \to R$ is a continuous function and defined by
\begin{equation}
F_{\alpha}(u,v)= \frac{\alpha}{2}[M(u,v) - W(u,v)]
\end{equation}
with $M(u,v)$ and $W(u,v)$ being the upper and lower Frechet bounds of the copula $C(u,v), \alpha \in [0,1]$. The function $F_{\alpha}$ is called a perturbation factor.

\begin{theorem}
Let $C(u, v)$ be any bivariate copula, defined by $C:[0,1]^2 \to [0,1]$. Then, a perturbed copula defined by $C_{\alpha}(u,v)= C(u,v) + F_{\alpha}(u,v)$ referred to as associated with $C(u, v)$, is a copula, where $\alpha \in [0,1].$
\end{theorem}
\begin{proof}
A proof is given in appendix.
\end{proof}

\subsubsection{Dependence properties of perturbed copula}
Now, we derive four measures of dependency, Spearman's rho, Kendall's tau, Blomqvist's beta and Gini's gamma of the perturbed copula:
$$C_{\alpha}(u,v)= C(u,v) + \frac{\alpha}{2}\Big[M(u,v) - W(u,v)\Big]$$.

\begin{theorem} The measures of dependency: Spearman's rho, Blomqvist's beta and the Gini's gamma of the perturbed copula
$$ C_{\alpha}(u,v)= C(u,v) + \frac{\alpha}{2}\Big[M(u,v) - W(u,v)\Big],$$
are given by $\rho(C)+ \alpha$,  $\beta(C)+ \alpha$, $\gamma(C)+ \alpha$, and the Kendall's tau is $\tau(C)+ 2\alpha \int_0^1\int_0^1\Big[M(u,v)- W(u,v)\Big]\frac{\partial ^2 C}{\partial u \partial v} dudv$.
\end{theorem}
\begin{proof}
A proof is given in appendix.
\end{proof}

\subsubsection{Examples}
In this sub-section, we provide examples of some copulas to define the measures of dependency for the perturbed copula.\\
\break
\textbf{For the product copula}\\
As the measures of dependency for the product copula $C(u,v) = uv, $ is 0, the measures of
dependency for the perturbed copula $ C_{\alpha}(u,v)= uv + \frac{\alpha}{2}\Big[\min(u,v) - \max(u+v-1, 0)\Big]$ are given by
\begin{align*}
\rho(C_{\alpha})&= \alpha, \quad \tau(C_{\alpha})= 2\alpha \int_0^1\int_0^1\Big[M(u,v) - W(u,v)\Big]\frac{\partial ^2 C}{\partial u \partial v} dudv,\\
\gamma(C_{\alpha})& = \alpha, \quad \textrm{and} \quad \beta(C_{\alpha})= \alpha.
\end{align*}
\break
\textbf{For the FGM copula}\\
The measures of dependency for the FGM copula $C(u,v) = uv  + \theta uv(1-u)(1-v)$, are given by
\begin{align*}
\rho(C)&= \frac{\theta}{3}, \quad \tau(C) = \frac{2\theta}{9}, \quad \beta(C) =  \frac{\theta}{4}, \quad \textrm{and} \quad \gamma(C)= \frac{4\theta}{15}.
\end{align*}
So, the four measures of dependency for the perturbed copula $ C_{\alpha}(u,v)= uv + \frac{\alpha}{2}\left[\min(u,v) - \max(u+v-1, 0)\right]$ are given by
\begin{align*}
\rho(C_{\alpha})&= \frac{\theta}{3} + \alpha, \quad \tau(C_{\alpha})=  \frac{2\theta}{9} + 2\alpha \int_0^1\int_0^1\Big[M(u,v) - W(u,v)\Big]\frac{\partial ^2 C}{\partial u \partial v} dudv , \\
\beta(C_{\alpha}) &=  \frac{\theta}{4} + \alpha, \quad \gamma(C_{\alpha}) =  \frac{4\theta}{15} + \alpha .
\end{align*}
\break
\textbf{For the Bound copula}\\
By using measures of dependencies of bound copula derived in theorem 2.2, the measures of dependency of perturbed copula are given by
\begin{align*}
\rho(C_{\alpha})&= \frac{\theta(\theta^2 + 1)}{2} + \alpha , \\
\tau(C_{\alpha}) &=   \frac{\theta(8 -3\theta + 12\theta^2 + 6\theta^3 + 4\theta^4 - 3\theta^5)}{24}  + 2\alpha \int_0^1\int_0^1\Big[M(u,v) -W(u,v)\Big] \frac{\partial ^2 C}{\partial u \partial v} dudv, \\
 \beta(C_{\alpha}) &=  \frac{\theta(\theta^2 + 1)}{2} +\alpha, \quad \textrm{and} \quad
\gamma(C_{\alpha}) =  \frac{\theta(\theta^2 + 1)}{2} + \alpha.
\end{align*}

\subsubsection{Tail dependence of perturbed copula}
In this sub-section, we derive the coefficient of the upper and the lower tail dependency of the perturbed copula:
$$C_{\alpha}(u,v)= C(u,v) + \frac{\alpha}{2}\Big[M(u,v) - W(u,v)\Big].$$

\begin{theorem}
The coefficients of the lower and the upper tail dependence of the perturbed copula $ C_{\alpha}(u,v)= C(u,v) + \frac{\alpha}{2}\Big[\min(u,v) - \max(u+v-1, 0)\Big]$ are given by \( \lambda_{L}^{C_{\alpha}} = \lambda_L^C + \frac{\alpha}{2}\) and
\( \lambda_{U}^{C_{\alpha}} = \lambda_U^C + \frac{\alpha}{2}.\), respectively.
\end{theorem}
\begin{proof}
A proof is given in appendix.
\end{proof}

\section{Dependency between two service times}
In this section, we consider two parallel queues with two dependent servers where two replicas of each task are launched with cancel-on-finish policy. Then, we analyze and evaluate the mean waiting time and the mean service time of a task. To reduce the waiting time, any task can try in more than one queue. One method to reduce the waiting time is the use of redundancy. Running a task on multiple machines and waiting for the earliest copy to finish can reduce the waiting time \cite{joshi2015efficient}.

\subsection{Waiting time and service time analysis}
Joshi, Soljanin and Wornell \cite{joshi2017efficient} analyzed the expected latency with the full replication and they used two policies, cancel on finish and cancel on start. Here, we consider the process with cancel on finish only because we can deal the dependence between service times using this policy. For cancel on start policy, it is difficult to use copula to analyze the dependency between the service times.  we consider two parallel queues with two parallel servers and deal with the situation where the servers are dependent.

With the cancel-on-finish policy, when any copy of the task has been served, the other copy is canceled and removed from the system immediately, or upon the service completion of any copy of the task, the other copy will be abandoned from the system immediately. The mean waiting time is the expected waiting time in the two queues, and the mean service time is defined as the expected total time spent by both servers on both replicas, since the service resources of the both servers have been allowed to the task.
\begin{lemma}
If we launch two copies for each task using the cancel-on-finish policy, the waiting time is equivalent in distribution to that of an  $ M/G/1$ queue with service time; $S_{1:2}= \min(S_1, S_2)$, where $S_i =$ Service time of server $ i$, $i = 1,2.$ That is, the expected waiting time and the expected service time are given by
$$ E(W) = E[W^{M/G/1}] = \frac{\lambda E(S_{1:2}^2)}{2(1- \lambda E(S_{1:2}))}, \quad and \quad E(S) = 2E(S_{1:2}), $$ respectively where $\lambda$ is the arrival rate of the task.
\end{lemma}
\begin{proof} A proof is given in the paper \cite{joshi2017efficient}.\\
\end{proof}

\subsection{Dependency between service times using the bound copula}
In many applications, servers are working independently, for example, in manufacturing, the speed of a machine is, in general, fixed and independent of other machines. However, in some other cases, service times can be dependent. For example, if servers are intelligent (say human being), peer's pressure can cause its service pair dependent. Even when servers are not human-being, competitions between different service providers can make service times non-independent. One server's pace motivates other servers to be slower or faster. If one server works at a relaxed pace, then other servers also work at a relaxed pace (slow) which is a positive dependence. If one server works at a relaxed pace, other servers increase their pace to maintain the system output which is a negative dependence between the servers. So, it is interesting to deal with the situation in which servers can be dependent on each other. To analyze dependency between the servers, we use the bound copula. For this method, we consider that the service times are shifted exponential and hypo exponential distributions, respectively.

\subsubsection{Service time is shifted exponential}
In this subsection, we estimate the mean waiting time and the mean service time for dependent service time and assume that the service time is shifted exponential distribution.
\begin{definition}
A continuous random variable $X$ is said to have the shifted exponential distribution with parameters $(\mu, \delta)$ if it has probability density function:
\begin{equation}
 f(x) =
 \begin{cases}
 \mu e^{-\mu(x-\delta)}, & \text{if $x \geq \delta$}\\
 0, & \text{if otherwise}\\
 \end{cases}
 \end{equation}
\end{definition}

Taking $\delta =0$ gives the pdf of the exponential distribution function. The cumulative distribution function $(CDF)$ of $X$ is given by $F_{X}(x) = 1 - e^{-\mu(x-\delta)},$ where $x \geq \delta$. Let the arrival rate be $\lambda$, and the service times $S_1$ and $S_2$ be shifted exponential with parameters $(\mu,\lambda)$. Then, the probability density function of service times $S_i$, $i$ =1, 2 is given by (3.1). \\
\break
\textbf{3.2.1.1 Estimating $E(W)$ and $E(S)$ for dependent service times}\\
When two replicas of each task are launched to the two servers by the cancel-on-finish policy, we have from lemma (3.1) that
\begin{align*}
E(W) &= E[W^{M/G/1}] = \frac{\lambda E(S_{1:2}^2)}{2\big[1- \lambda E(S_{1:2})\big]},
\end{align*}
and $E(S) = 2E(S_{1:2}),$ where $S_{1:2} = \min(S_1, S_2).$ Now,
$$ P(S_{1:2} > x) = P(\min(S_1, S_2) > x) = P(S_1 > x, S_2 > x).$$
\break
\textbf{Case(i)} If the service times $S_1$ and $S_2$ are independent,
\begin{align*}
P(S_{1:2} > x) &= P(S_1 > x, S_2 > x)= P(S_1 > x)(S_2 > x) \nonumber \\
&= e^{-\mu(x-\delta)}.e^{-\mu(x-\delta)}= e^{-2\mu(x-\delta)}.
\end{align*}
So, probability distribution function of $S_{1:2}$ is given by, \\
\begin{align*}
P(S_{1:2} \leq x)= F_{X_{1:2}}(x) = 1 - e^{-2\mu(x-\delta)},
\end{align*}
and its probability density function is given by
\begin{align*}
f_{S_{1:2}}(x)= 2\mu e^{-2\mu(x-\delta)}, x \geq \delta.
\end{align*}
Now,
$$E(S_{1:2}) = \int_{\delta}^{\infty} x f_{S_{1:2}}(x)dx = \int_{\delta}^{\infty}2\mu x e^{-2\mu(x-\delta)}dx = \delta + \frac{1}{2\mu} $$ and
\begin{align*}
E(S_{1:2}^2) &= \int_{\delta}^{\infty} x^2 f_{S_{1:2}}(x)dx = \int_{\delta}^{\infty}2\mu x^2 e^{-2\mu(x-\delta)}dx = \delta^2 + \frac{\delta}{\mu} + \frac{1}{2\mu^2}. \\
Thus,~{} E(W) &= E[W^{M/G/1}] = \frac{\lambda E(S_{1:2}^2)}{2(1- \lambda E(S_{1:2}))}= \frac{\lambda (\delta^2 + \frac{\delta}{\mu} + \frac{1}{2\mu^2})}{2[1- \lambda(\delta + \frac{1}{2\mu})]},
\end{align*}
Then, the mean service time is given by
$$ E(S) = 2E(S_{1:2})= 2\Big(\delta + \frac{1}{2\mu}\Big) = 2\delta + \frac{1}{\mu}. $$
\break
\textbf{Example:} \\
If we choose $\lambda = 0.25$, $\mu = 0.5$, $\rho = \frac{\lambda}{\mu} = 0.5$ and  $\delta= 1$, we get, $E(W) = 1.25$ and $E(S) = 4$. \\
\break
\textbf{Case(ii)} For service times $S_1$ and $S_2$ are dependent, consider $ P(S_{1:2} > x) = P(S_1 > x, S_2 > x)$. Let $u = F_{S_1}(x)= 1 - e^{-\mu(x-\delta)}$  and $v = F_{S_2}(x)= 1 - e^{-\mu(x-\delta)}$. Copulas are used to describe the dependence between the two service times $S_1$ and $S_2$ and we use Bound copula, which is defined as $C(u,v) = (1-\theta^2)uv + \frac{\theta}{4}\Big[(1+\theta)^2 \min(u,v) - (1-\theta)^2 \max(u+v-1,0)\Big]$, where $\theta \in [-1,1]$ is a dependence parameter.

Now, the joint survival function of $S_1$ and $S_2$ is given by
\begin{align*}
P(S_{1:2} > x)&= P(S_1 > x, S_2 > x)= \bar{C}(1-u, 1-v) = 1 - u - v + C(u,v)
\end{align*}
Then, for $u+v \geq 1$, the distribution function of $S_{1:2} = \min(S_1,S_2)$ is given by
\begin{align*}
F_{S_{1:2}}(x)&= 2 - 2e^{-\mu(x-\delta)} - (1-\theta^2)\Big[1 - 2e^{-\mu(x-\delta)}+ e^{-2\mu(x-\delta)}\Big]- \frac{\theta}{4}\Big[(1+\theta)^2 \\
&\quad(1 - e^{-\mu(x-\delta)}) -(1-\theta)^2(1- e^{-\mu(x-\delta)}+ 1 - e^{-\mu(x-\delta)}- 1)\Big].
\end{align*}
and the density function of $S_{1:2}(x)$ is given by
\begin{align*}
f_{S_{1:2}}(x)&= 2\mu e^{-\mu(x-\delta)}-(1-\theta^2)\Big[2\mu e^{-\mu(x-\delta)}- 2\mu e^{-2\mu(x-\delta)}\Big]- \frac{\theta}{4}\Big[(1+\theta)^2 \mu e^{-\mu(x-\delta)}\Big]\\
 &\quad + \frac{\theta}{4}\Big[(1-\theta)^2 2\mu e^{-\mu(x-\delta)}\Big].
\end{align*}
Then, the mean of $S_{1:2}$ is given by
\begin{align*}
E(S_{1:2})&= \int_{\delta}^{\infty} xf_{S_{1:2}}(x)dx \\
E(S_{1:2})&= 2\mu\left(\frac{\delta}{\mu}+ \frac{1}{\mu^2}\right)- (1-\theta^2)\left[2\mu\left(\frac{\delta}{\mu}+ \frac{1}{\mu^2}\right)- 2\mu\left(\frac{\delta}{2\mu}+ \frac{1}{4\mu^2}\right)\right] \\
 &\quad -\frac{\theta}{4}(1+\theta)^2 \mu \left(\frac{\delta}{\mu}+ \frac{1}{\mu^2}\right)+ \frac{\theta}{4}(1 -\theta)^2 2\mu \left(\frac{\delta}{\mu}+ \frac{1}{\mu^2}\right),~{}~{} and
\end{align*}
\begin{align*}
E(S_{1:2}^2) &=  \int_{\delta}^{\infty} x^2 f_{S_{1:2}}(x)dx \\
&= 2\mu\left(\frac{\delta^2}{\mu}+ \frac{2\delta}{\mu^2}+ \frac{2}{\mu^3}\right)- (1-\theta^2)\bigg[2\mu\left(\frac{\delta^2}{\mu}+ \frac{2\delta}{\mu^2}+ \frac{2}{\mu^3}\right)-2\mu\Big(\frac{\delta^2}{2\mu}+ \frac{2\delta}{4\mu^2}+  \\
&\quad \frac{2}{8\mu^3}\Big)\bigg]- \frac{\theta}{4}(1+\theta)^2 \mu\left(\frac{\delta^2}{\mu}+ \frac{2\delta}{\mu^2}+ \frac{2}{\mu^3}\right)+ \frac{\theta}{4}(1- \theta)^2 2\mu\left(\frac{\delta^2}{\mu}+ \frac{2\delta}{\mu^2}+ \frac{2}{\mu^3}\right).
\end{align*}
\break
\textbf{Example:}\\
Using $\delta = 1$ and $\mu = 0.5$, we have after calculations,
$$ E(S_{1:2}) = 2 + 0.75\theta - 0.5 \theta^2 + 0.75 \theta^3 ~{}~{}and ~{}~{}E(S_{1:2}^2) = 5 + 3.25\theta + 1.5\theta^2 + 3.25\theta^3.$$
Therefore, the values of $E(W)$ and $E(S)$ are given by \\
$$ E(W) = \frac{\lambda E(S_{1:2}^2)}{2[1 - \lambda E(S_{1:2})]}= \frac{1.25+0.8125\theta+ 0.375\theta^2 + 0.8125\theta^3}{1- 0.375\theta+ 0.25\theta^2- 0.375\theta^3} ~{}~{}and $$
$$ E(S) = 2E(S_{1:2}) = 4 + 1.5\theta - \theta^2 + 1.5\theta^3. $$
The values of $E(W)$ and $E(S)$ for positive and negative measures of dependency ($\theta$) are shown in tables (2) and (3). The line graphs of $E(W)$ and $E(S)$ corresponding to positive and negative dependence are shown in figure 2.  \\

\begin{table}[hbt!]
\begin{center}
\maketitle
\begin{tabular}{|c|c|c|c|c|c|c|}
\hline
$\theta$ & 0 & 0.1 & 0.2 & 0.3 & 0.4 & 0.5\\
\hline
E(W) & 1.25 & 1.38 & 1.54 & 1.72 & 1.94 & 2.23\\
\hline
E(S) & 4 & 4.14 & 4.27 & 4.40 & 4.53 & 4.68\\
\hline
\end{tabular}
\caption{\label{tab: table 4.4} Table of $E(W)$ and $E(S)$ vs positive dependence}
\end{center}
\end{table}

\begin{table}[hbt!]
\begin{center}
\maketitle
\begin{tabular}{|c|c|c|c|c|c|}
\hline
$\theta$ & -0.1 & -0.2 & -0.3 & -0.4 & -0.5\\
\hline
E(W) & 1.12 & 1.00 & 0.89 & 0.77 & 0.64\\
\hline
E(S) & 3.83 & 3.64 & 3.42 & 3.14 & 2.81\\
\hline
\end{tabular}
\caption{\label{tab: table 4.5} Table of $E(W)$ and $E(S)$ vs negative dependence}
\end{center}
\end{table}

\begin{figure}[hbt!]
\centering
\caption{Line graphs of $E(W)$ and $E(S)$ vs dependency}
\label{Graph 4.1}
\includegraphics[scale=.7]{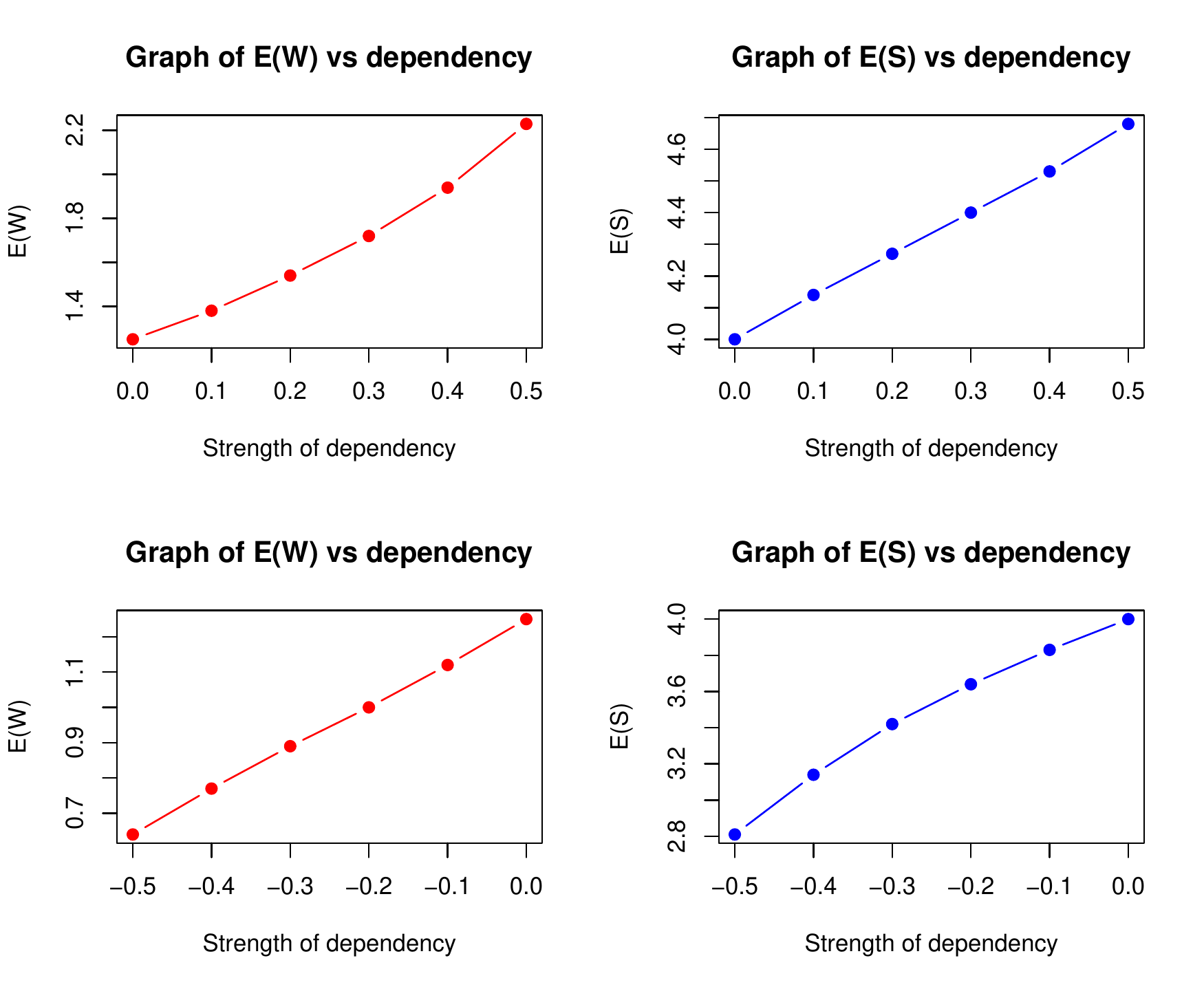}
\end{figure}
We observe that, for increasing positive dependency, the values of $E(W)$ and $E(S)$ also increase while for increasing negative dependency (decreasing value of $\theta$),  the values of $E(W)$ and $E(S)$ decrease.\\

\subsubsection{Service time is hypo exponential}
In this subsection, we estimate the mean waiting time and the mean service time for dependent service time and assume that the service time is hypo exponential distribution.

\begin{definition}
A continuous random variable $X$ is said to have hypo exponential distribution if its probability density function is defined as
\begin{equation}
f(x) =
\begin{cases}
\mu^2 x e^{-\mu x}, & \text{if $x \geq 0$}\\
0, & \text{ otherwise}\\
\end{cases}
\end{equation}
\end{definition}

and its distribution function is given as
$$ F(x) = 1 - e^{-\mu x} - \mu x e^{-\mu x}, \quad x \geq 0. $$
\break
\textbf{3.2.2.1 Evaluating $E(W)$ and $E(S)$} \\
When two replicas of a task launches to two servers using the cancel on finish policy, we have (see lemma 5.1),
$$ E(W) = \frac{\lambda E(S_{1:2}^2)}{2\big[1- \lambda E(S_{1:2})\big]} ~{} and ~{}
 E(S) = 2E(S_{1:2}),$$ where $S_{1:2} = \min(S_1, S_2) $. Now,
$$ P(S_{1:2} > x) = P(\min(S_1, S_2) > x) = P(S_1 > x, S_2 > x).$$
\break
\textbf{Case(i)} If service times $S_1$ and $S_2$ are independent, then
\begin{align*}
P(S_{1:2} > x) &= P(S_1 > x, S_2 > x)= P(S_1 > x)(S_2 > x), \\
&= (e^{-\mu x} + \mu x e^{-\mu x})^2 = e^{-2\mu x} + 2\mu x e^{-2\mu x}+ \mu^2 x^2 e^{-2\mu x}.
\end{align*}
Its distribution function is given by
\begin{align*}
P(S_{1:2} \leq x)= 1 - e^{-2\mu x} - 2\mu x e^{-2\mu x} - \mu^2 x^2 e^{-2\mu x}.
\end{align*}
and its probability density function is given by
\begin{align*}
f_{S_{1:2}}(x) = 2\mu^2 x e^{-2\mu x} + 2\mu^3 x^2 e^{-2\mu x}.
\end{align*}
Then, after calculations, we get
\begin{align*}
E(S_{1:2}) = \int_{0}^{\infty} x f_{S_{1:2}}(x)dx = \int_{0}^{\infty}(2\mu^2 x^2 e^{-2\mu x} + 2\mu^3 x^3 e^{-2\mu x})dx = \frac{5}{4\mu} ~{}~{}and
\end{align*}
\begin{align*}
E(S_{1:2}^2) = \int_{0}^{\infty} x^2 f_{S_{1:2}}(x)dx = \int_{0}^{\infty}(2\mu^2 x^3 e^{-2\mu x} + 2\mu^3 x^4 e^{-2\mu x})dx = \frac{9}{4\mu^2}.
\end{align*}
Then, the mean waiting time and the mean service time become
$$ E(W) = \frac{\lambda E(S_{1:2}^2)}{2\big[1- \lambda E(S_{1:2})\big]}= \frac{9\lambda}{2\mu(4\mu- 9\lambda)}~{}and ~{} E(S) = 2 \times E(S_{1:2}) = \frac{5}{2\mu}.$$
\break
\textbf{Example:}\\
If we choose $\lambda= 0.125$, $\mu= 0.5$, then $ E(W) =  0.818 ~{}~{}and ~{}~{} E(S) = \frac{5}{2\mu}= 5. $ \\
\break
\textbf{Case (ii)} When service times $S_1$ and $S_2$ are dependent, we have\\
$ P(S_{1:2} > x)= P(\min(S_1, S_2) > x)= P(S_1 > x, S_2 > x).$ \\
Let $u = F_{S_1}(x)= 1 - e^{-\mu x} - \mu x e^{-\mu x}$ and $v = F_{S_2}(x)= 1 - e^{-\mu x} - \mu x e^{-\mu x} $, and take $u + v \geq 1$.\\
Bound copula is given by \\
$C(u,v) = (1-\theta^2)uv + \frac{\theta}{4}\Big[(1+\theta)^2 \min(u,v) - (1-\theta)^2 \max(u+v-1,0)\Big]$, where $\theta \in [-1,1]$ is dependence parameter.\\
Now, the joint survival function of $S_1$ and $S_2$ is given by
\begin{align*}
P(S_{1:2} > x)&= P(S_1 > x, S_2 > x)= \bar{C}(1-u, 1-v) = 1 - u - v + C(u,v),
\end{align*}
Distribution function of $S_{1:2}$ is given by
\begin{align*}
P(S_{1:2}\leq x)&= F_{S_{1:2}}(x) = 2 - 2e^{-\mu x}- 2\mu x e^{-\mu x}- (1- \theta^2)(1- 2e^{-\mu x}\\
&\quad - 2\mu x e^{-\mu x}+ e^{-2\mu x}+ 2\mu x e^{-2\mu x}+ \mu^2 x^2 e^{-2\mu x})-\frac{\theta}{4}(1+\theta)^2(1 \\
&\quad -e^{-\mu x} - \mu x e^{-\mu x})+ \frac{\theta}{4}(1-\theta)^2(1-2e^{-\mu x}- 2\mu x e^{-\mu x}).
\end{align*}
and the density function of $S_{1:2}$ is given by
\begin{align*}
f_{S_{1:2}}(x) &= 2\mu^2 x e^{-\mu x}-(1- \theta^2)(2\mu^2 x e^{-\mu x}- 4\mu^2 x e^{-2\mu x}+ 2\mu^2 x e^{-2\mu x}\\
&\quad - 2\mu^3 x^2 e^{-2\mu x})- \frac{\theta(1+\theta)^2}{4}\mu^2 x e^{-\mu x}+ \frac{\theta(1-\theta)^2}{2}\mu^2 x e^{-\mu x}.
\end{align*}
Then, the mean of $S_{1:2}$ is given by
\begin{align*}
E(S_{1:2})&= \int_{0}^{\infty} x f_{S_{1:2}}(x)dx, \\
&= \int_{0}^{\infty}[2\mu^2 x^2 e^{-\mu x}-(1- \theta^2)(2\mu^2 x^2 e^{-\mu x}- 4\mu^2 x^2 e^{-2\mu x}+ 2\mu^2 x^2 e^{-2\mu x} \\
&\quad - 2\mu^3 x^3 e^{-2\mu x})- \frac{\theta}{4}(1+\theta)^2 \mu^2 x^2 e^{-\mu x}+ \frac{\theta}{2}(1-\theta)^2 \mu^2 x^2 e^{-\mu x}]dx.
\end{align*}
After calculations, we get
$$ E(S_{1:2})= \frac{4}{\mu}- (1- \theta^2)\left[\frac{4}{\mu}- \frac{1}{\mu}+ \frac{1}{2\mu}- \frac{0.75}{\mu}\right]- \frac{\theta(1+\theta)^2}{2\mu}+ \frac{\theta(1-\theta)^2}{\mu}.$$
We take $\mu = 0.5$, we then have, $ E(S_{1:2})= 2.5 + \theta - 0.5\theta^2 + \theta^3.$ \\
The second moment of $S_{1:2}$ is given by
\begin{align*}
E(S_{1:2}^2)&= \int_{0}^{\infty} x^2 f_{S_{1:2}}(x)dx,\\
&= \int_{0}^{\infty}\Big[2\mu^2 x^3 e^{-\mu x}- (1-\theta^2)(2\mu^2 x^3 e^{-\mu x}- 4\mu^2 x^3 e^{-2\mu x}+ 2\mu^2 x^3 e^{-2\mu x}\\
&\quad - 2\mu^3 x^4 e^{-2\mu x})-\frac{\theta(1+\theta)^2}{4} \mu^2 x^3 e^{-\mu x}+ \frac{\theta(1-\theta)^2}{2} \mu^2 x^3 e^{-\mu x}\Big]dx.
\end{align*}
After calculations, we get
$$ E(S_{1:2}^2)= \frac{12}{\mu^2} - (1-\theta^2)\Big( \frac{12}{\mu^2}- \frac{3}{\mu^2}+  \frac{3}{4\mu^2}- \frac{3}{2\mu^2}\Big)- \frac{3\theta(1+\theta)^2}{2\mu^2}+ \frac{3\theta(1-\theta)^2}{\mu^2}.$$
\break
\textbf{Example:}\\

If $\mu = 0.5$, $ E(S_{1:2}^2)= 9 + 6\theta + 3\theta^2 + 6\theta^3. $ Then, the mean waiting time and mean service time are given by
$$ E(W) = \frac{\lambda E(S_{1:2}^2)}{2\big(1- \lambda E(S_{1:2})\big)}= \frac{0.125(9+6\theta+3\theta^2+6\theta^3)}{2-0.25(2.5+\theta -0.5\theta^2+ \theta^3)}, $$
and $$ E(S) = 2E(S_{1:2})= 5+2\theta-\theta^2 + 2\theta^3. $$
The values of $E(W)$ and $E(S)$ for positive and negative strengths of dependencies are shown in tables (6)and (7). Also, the line graphs of $E(W)$ and $E(S)$ for different strength of dependencies are shown in figure (4).

\begin{figure}[hbt!]
\centering
\caption{Line graphs of $E(W)$ and $E(S)$ vs dependency}
\label{Graph 4.4}
\includegraphics[scale=.85]{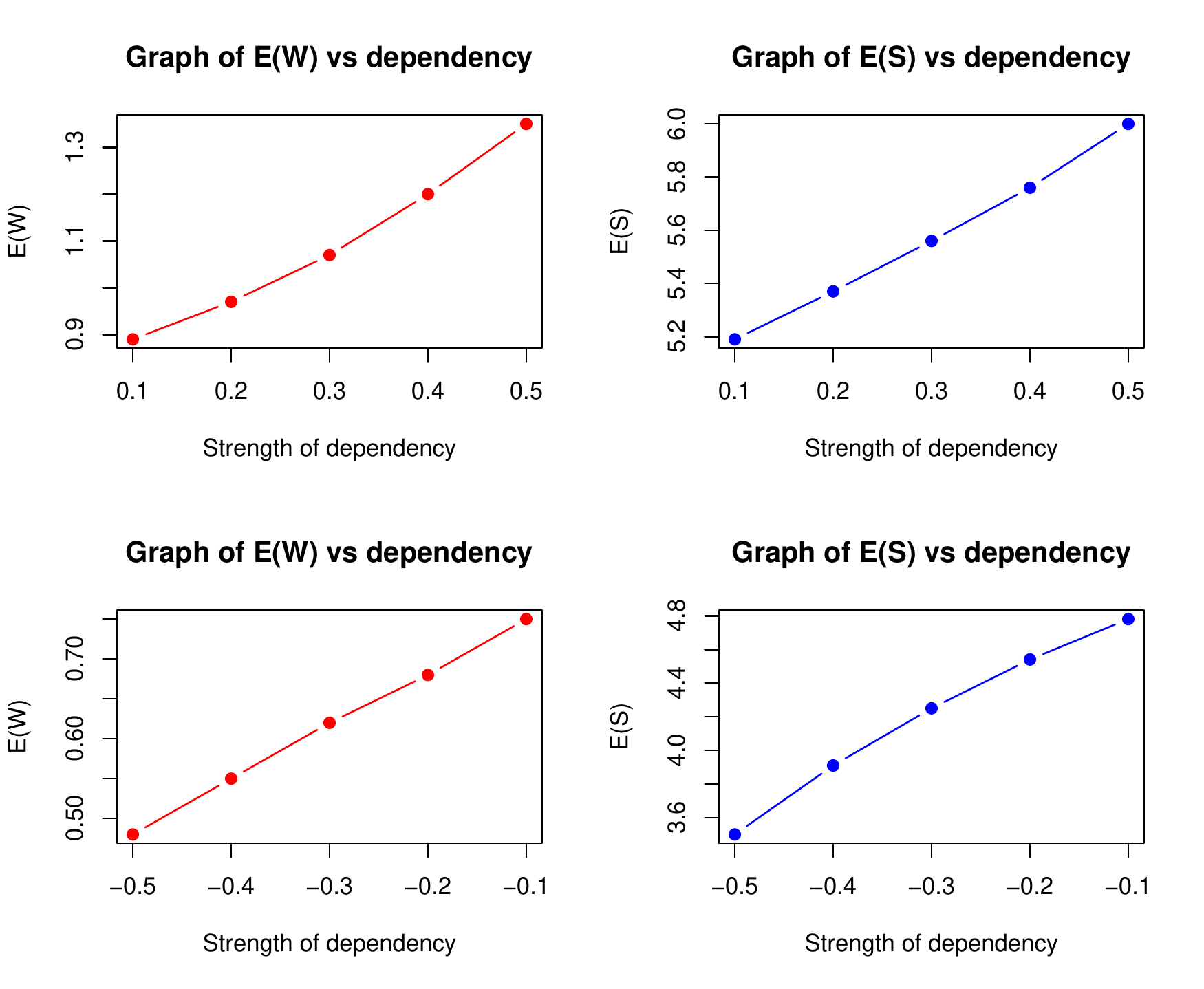}
\end{figure}

\begin{table}[hbt!]
\begin{center}
\maketitle
\begin{tabular}{|c|c|c|c|c|c|}
\hline
$\theta$ & 0.1 & 0.2 & 0.3 & 0.4 & 0.5\\
\hline
E(W) & 0.89 & 0.97 & 1.07 & 1.20 & 1.35\\
\hline
E(S) & 5.19 & 5.37 & 5.56 & 5.76 & 6\\
\hline
\end{tabular}\\
\caption{\label{tab: table 4.9} Table of $E(W)$ and $E(S)$ vs positive dependence}
\end{center}
\end{table}

\begin{table}[hbt!]
\begin{center}
\maketitle
\begin{tabular}{|c|c|c|c|c|c|}
\hline
$\theta$ & -0.1 & -0.2 & -0.3 & -0.4 & -0.5\\
\hline
E(W) & 0.75 & 0.68 & 0.62 & 0.55 & 0.48\\
\hline
E(S) & 4.78 & 4.54 & 4.25 & 3.91 & 3.50\\
\hline
\end{tabular}\\
\caption{\label{tab: table 4.10}Table of $E(W)$ and $E(S)$ vs negative dependence}
\end{center}
\end{table}

We observe that for positive dependencies, as measures of dependency increases between service times, the values of $E(W)$ and $E(S)$ also increase and for negative measures of dependency, values of $E(W)$ and $E(S)$ decrease.

\section{Conclusion}
In Section 2, we constructed two new copulas, the bound  copula and a perturbed copula. We also derived measures of dependency such as Spearman's rho, Kendall's tau, Blomqvist's beta and Gini's gamma along with coefficients of the upper tail and the lower tail dependencies.

For application, in Section 3, we estimated the mean waiting time and the mean service time  when two replicas of a task launch to two parallel dependent servers. We have shown that for positive dependency between service times, when the strength of dependency increases, the mean waiting time and the mean service time also increase whereas for the negative dependent case, they decrease. To generalize the result, we have chosen two different service times and got the similar results for both cases. To analyze the dependency between the service times, we have used the newly constructed bound copula, which is suitable for our application for analyzing service times. \\
\break
\textbf{Acknowledgement}

The authors would like to acknowledge that this work is supported, in part, by the Natural Sciences and Engineering Research Council of Canada (NSERC) through a Discovery Research Grant, and Carleton University.\\
\break
\textbf{Appendix}\\
\break
\textbf{Proof of theorem 2.1}

For $\theta=0 $, we have,  $C(u,v)=uv$, which implies independence.
When $\theta=1$, simple calculation leads to
\begin{align*}
C(u,v) = \frac{1}{4}\left[2^2 \min(u,v) - 0\right]= \min(u,v).
\end{align*}
Similarly, when $\theta = -1$, it becomes
\begin{align*}
C(u,v) = -\frac{1}{4}\left[-4\max(u+v-1,0)\right] = \max(u+v-1, 0).
\end{align*}
We now need to show for all $\theta \in [-1,1]$, $C(u,v)$  satisfies all properties of a copula.
\begin{description}
\item[(i)] Boundary conditions:
\begin{align*}
C(u,0)&= \frac{\theta}{4}[(1+\theta)^2.0 - (1-\theta)^2 \cdot 0]= 0 = C(0,v),\\
C(u,1) &= (1-\theta ^2)u + \frac{\theta}{4}[(1+\theta)^2 \min(u,1) - (1-\theta)^2 \max(u+1-1,0)] \\
&= (1-\theta ^2)u + \frac{\theta}{4}[(1+\theta)^2.u - (1-\theta)^2 \cdot u] = u - \theta^2 u + \theta^2 u = u.
\end{align*}
Similarly, $ C(1,v) = v. $

\item[(ii)] 2-increasing property: For all $0 \leq u_1 \leq u_2 \leq 1$, $0 \leq v_1 \leq v_2 \leq 1,$
\end{description}
 $C(u_1, v_1) + C(u_2,v_2) - C(u_1,v_2)-C(u_2, v_1)$
\begin{align}
&=(1-\theta^2)u_1 v_1 + \frac{\theta}{4}\Big[(1+\theta)^2 \min(u_1,v_1)-(1-\theta)^2 \max(u_1+v_1-1,0)\Big]+ (1-\theta^2)\nonumber \\
&\quad u_2 v_2 + \frac{\theta}{4}\Big[(1+\theta)^2 \min(u_2,v_2) -(1-\theta)^2 \max(u_2+v_2-1,0)\Big]-(1-\theta^2)u_1 v_2 \nonumber \\
&\quad - \frac{\theta}{4}\Big[(1+\theta)^2 \min(u_1,v_2) - (1-\theta)^2 \max(u_1+v_2-1,0)\Big]-(1-\theta^2)u_2 v_1 \nonumber\\
&\quad - \frac{\theta}{4}\Big[(1+\theta)^2 \min(u_2,v_1) - (1-\theta)^2 \max(u_2+v_1-1,0)\Big]\nonumber \\
&= (1-\theta^2)(u_2-u_1)(v_2-v_1) + \frac{\theta}{4}\big(1+\theta\big)^2\Big[\min(u_1,v_1) + \min(u_2,v_2)- \min(u_1,v_2)\nonumber \\
&\quad - \min(u_2,v_1)\Big]- \frac{\theta}{4}(1 -\theta)^2\Big[\max(u_1+v_1-1, 0) + \max(u_2+ v_2-1, 0)\nonumber \\
&\quad -\max(u_1+v_2-1, 0)- \max(u_2+v_1-1, 0)\Big].
\end{align}
(a) For $u \leq v$, $u + v \leq 1$, that is, for $u_1 \leq v_1$, $u_1 + v_1 \leq 1$, $u_2 \leq v_2$, $u_2 + v_2 \leq 1$,\\
the expression (4.1) becomes \\ $C(u_1, v_1) + C(u_2,v_2) - C(u_1,v_2)-C(u_2, v_1)$
\begin{align*}
&=(1-\theta^2)(u_2-u_1)(v_2-v_1) + \frac{\theta}{4}\big(1+\theta\big)^2\Big[u_1 + u_2 - u_1 - \min(u_2, v_1)\Big] - \frac{\theta}{4}\big(1-\theta\big)^2 \cdot 0] \\
&= (1-\theta^2)(u_2-u_1)(v_2-v_1) + \frac{\theta}{4}\big(1+\theta\big)^2\Big[u_2 - \min(u_2,v_1)\Big].
\end{align*}
If $u_2 \leq v_1$, then $u_2 - \min(u_2,v_1) = u_2 - u_2 = 0.$ \\
So, the above expression becomes
$(1-\theta^2)(u_2-u_1)(v_2-v_1) \geq 0$, $-1 \leq \theta \leq 1$, $u_2 \geq u_1$ and  $v_2 \geq v_1.$ \\
If $u_2 > v_1$, then
$$ \Big[(1-\theta^2)(u_2-u_1)(v_2-v_1) + \frac{(1 + \theta)(\theta^2 + \theta)}{4}(u_2 - v_1)\Big] \geq 0,$$
since  $1 + \theta \geq 0$,~{} $\theta^2 + \theta \geq 0$ ~{}and~{} $1-\theta^2 \geq 0$ for $ -1 \leq \theta \leq 1. $ \\

(b) For $u \leq v$, $ u + v > 1$, that is, for $u_1 \leq v_1$, $u_1 + v_1 > 1$, $ u_2 \leq v_2$ and  $u_2 + v_2 > 1$, expression (4.1) becomes \\
$C(u_1, v_1) + C(u_2,v_2) - C(u_1,v_2)-C(u_2, v_1)$
\begin{align*}
&=(1-\theta^2)(u_2-u_1)(v_2-v_1) + \frac{\theta}{4}(1+\theta)^2\Big[u_1 + u_2 - u_1 - \min(u_2, v_1)\Big] \\
&\quad -\frac{\theta}{4}(1-\theta)^2\Big[u_1+v_1-1+u_2+v_2-1-u_1-v_2+1-u_2-v_1+1\Big] \\
&=(1-\theta^2)(u_2-u_1)(v_2-v_1) + \frac{\theta}{4}(1+\theta)^2\Big[u_2 -\min(u_2,v_1)\Big] \geq 0.
\end{align*}
(c) For $u>v$, $u+v \leq 1$, that is, for $u_1 > v_1$, $u_1 + v_1 \leq 1$, $u_2 > v_2$ and $u_2 + v_2 \leq 1,$ expression \\
(4.1) becomes \\
$C(u_1, v_1) + C(u_2,v_2) - C(u_1,v_2)-C(u_2, v_1)$
\begin{align*}
&=(1-\theta^2)(u_2-u_1)(v_2-v_1) + \frac{\theta}{4}(1+\theta)^2\Big[v_1 + v_2 - \min(u_1, v_2)-v_1\Big] - \frac{\theta}{4}(1-\theta)^2 \cdot 0 \\
&= (1-\theta^2)(u_2-u_1)(v_2-v_1) + \frac{(1 + \theta)(\theta^2 + \theta)}{4}\Big[v_2 -\min(u_1,v_2)\Big] \geq 0.
\end{align*}
Since $v_2-\min(u_1, v_2)\geq 0$, $1 + \theta \geq 0$, $\theta^2 + \theta \geq 0$,  $1-\theta^2 \geq 0$ for $ -1 \leq \theta \leq 1.$ \\

(d) For $u > v$, $u + v > 1 $, that is, for $u_1 > v_1$ and $u_1 + v_1 > 1$, the expression (4.1)\\ becomes

$C(u_1, v_1) + C(u_2,v_2) - C(u_1,v_2)-C(u_2, v_1)$
\begin{align*}
&= (1-\theta^2)(u_2-u_1)(v_2-v_1) + \frac{\theta}{4}(1+\theta)^2\Big[v_1 + v_2 - \min(u_1, v_2)- v_1\Big]\\
&\quad - \frac{\theta}{4}(1-\theta)^2\Big[u_1+v_1-1+u_2+v_2-1-u_1-v_2+1-u_2-v_1+1\Big] \\
&= \Big[(1-\theta^2)(u_2-u_1)(v_2-v_1) + \frac{(1 + \theta)(\theta^2 + \theta)}{4}\big(v_2 -\min(u_1, v_2\big)\Big] \geq 0,\\
& since ~{}~{} [v_2 - \min(u_1,v_2)] \geq 0.
\end{align*}
So, $C(u,v)$ is 2-increasing which implies that
$$C(u,v) = (1-\theta^2)uv + \frac{\theta}{4}\Big[(1+\theta)^2 \min(u,v) - (1-\theta)^2 \max(u+v-1,0)\Big]$$ is a copula.\\
\break
\textbf{Proof of theorem 2.2}\\

(i) By the definition of Spearman's rho, we have
\begin{align*}
\rho(C)&= 12 \int_0^1\int_0^1 C(u,v)\quad dudv - 3 \\
&= 12 \int_0^1\int_0^1\left[(1-\theta^2)uv + \frac{\theta}{4}\left((1+\theta)^2 \min(u,v) - (1-\theta)^2 \max(u+v-1,0)\right)\right]dudv - 3.
\end{align*}
\begin{align*}
\text{Since} \quad \int_0^1\int_0^1 uv dudv = \frac{1}{4}, \quad \int_0^1\int_0^1 \min(u,v)dudv = \frac{1}{3} \quad \text{and} \quad \int_0^1\int_0^1 \max(u+v-1,0)\quad dudv &= \frac{1}{6},
\end{align*}
\begin{align*}
\rho(C) &= 12\Big[(1-\theta^2)\frac{1}{4} + \frac{\theta}{4}\Big((1+\theta)^2 \frac{1}{3} - (1-\theta)^2 \frac{1}{6}\Big)\Big] - 3 = \frac{\theta(\theta^2 + 1)}{2}.
\end{align*}

(ii) By the definition of Kendall's tau, we have
\begin{equation}
\tau(C) = 1 - 4\int_0^1\int_0^1 \frac{\partial C}{\partial u}\cdot \frac{\partial C}{\partial v} dudv.
\end{equation}
We have, $$C(u,v) = (1- \theta^2)uv + \frac{\theta}{4}\Big[(1+\theta)^2 \min(u,v) - (1-\theta)^2 \max(u+v-1,0)\Big].$$
Then,
$$ \frac{\partial C}{\partial u} = (1- \theta^2)v + \frac{\theta(1+\theta)^2}{4} I_{[u\leq v]} - \frac{\theta(1-\theta)^2}{4} I_{[u\geq 1-v]},$$
and
$$\frac{\partial C}{\partial v} = (1- \theta^2)u + \frac{\theta(1+\theta)^2}{4}I_{[u > v]} - \frac{\theta(1-\theta)^2}{4} I_{[u\geq 1-v]}.$$
So,
\begin{align}
\frac{\partial C}{\partial u}\cdot \frac{\partial C}{\partial v}&= (1- \theta^2)^{2} uv + \frac{\theta(1+\theta)^2 (1-\theta^2)}{4} v I_{[u > v]}- \frac{\theta(1- \theta)^2 (1-\theta^2)}{4} I_{[u \geq 1-v]}\nonumber \\
&\quad + \frac{\theta(1+\theta)^2 (1-\theta^2)}{4}  u I_{[u \leq v]} + \frac{\theta^2(1+\theta)^4}{16} I_{[u \leq v]} I_{[u > v]} \nonumber \\
&\quad - \frac{\theta^2(1+\theta)^2 (1-\theta)^2}{16} I_{[u \leq v]}I_{[u \geq 1-v]} - \frac{\theta(1-\theta)^2 (1-\theta^2)}{4} u I_{[u \geq  1-v]} \nonumber \\
&\quad - \frac{\theta^2(1+\theta)^2 (1-\theta)^2}{16} I_{[u \geq v]}I_{[v \geq 1-u]}+ \frac{\theta^2(1-\theta)^4}{16} I_{[u \geq 1-v]}I_{[u \geq 1-v]}.
\end{align}
We have,
\begin{align}
\int_0^1\int_0^1 uv~{}dudv &= \frac{1}{4},~{}~{}\int_0^1\int_0^1 v I_{[u \geq v]}~{}dudv = \int_0^1\int_v^1 v~{} du dv = \frac{1}{6},\\
\int_0^1\int_0^1 v I_{[u \geq 1-v]}~{} dudv &= \int_0^1\int_{1-v}^1 v ~{} du dv  = \frac{1}{3},~{}~{}~{}\int_0^1\int_0^1 u I_{[u \leq v]}~{}dudv= \int_0^1\int_0^v u ~{}du dv = \frac{1}{6},
\end{align}
\begin{align}
\int_0^1\int_0^1 I_{[u \leq v]} I_{[u > v]}~{} dudv &= 0,~{}~{}\int_0^1\int_0^1 I_{[u \leq v]} I_{[u \geq 1-v]}~{}dudv = \int_0^1\int_{1-v}^v~{}du dv = 0,\\
\int_0^1\int_0^1 u I_{[u \geq 1-v]} dudv &= \int_0^1\int_{1-v}^1 u ~{}du dv = \frac{1}{3},~{}~{}\int_0^1\int_0^1 I_{[u \geq v]} I_{[u \geq 1-v]}~{}dv du=\\
&\quad \int_0^1\int_{1-u}^u dv du = 0, ~{}~{} and ~{}~{} \int_0^1\int_0^1 I_{[u \geq 1-v]} dudv = \int_0^1\int_{1-v}^1 ~{} du dv  = \frac{1}{2}.
\end{align}
Using the equations (4.3) to (4.8) in (4.2), equation (4.2) becomes
\begin{align*}
\tau(C)&= 1 - 4\left[\frac{(1-\theta^2)^2}{4} + \frac{\theta(1-\theta^2)(1+ \theta)^2}{24}-  \frac{\theta(1-\theta^2)(1- \theta)^2}{12}\right]  \\
&\quad - 4\left[\frac{\theta(1-\theta^2)(1+ \theta)^2}{24}- \frac{\theta(1-\theta^2)(1- \theta)^2}{12}+ \frac{\theta^2 (1- \theta)^4}{32}\right]\\
\tau(C)&= 1- 4\left[\frac{(1-\theta^2)^2}{4}+ \frac{\theta(1-\theta^2)(1+ \theta)^2}{12}- \frac{\theta(1-\theta^2)(1- \theta)^2}{6} + \frac{\theta^2 (1- \theta)^4}{32}\right].
\end{align*}
After calculations, we get
\begin{equation}
\tau(C) = \frac{\theta(8 -3\theta + 12\theta^2 + 6\theta^3 + 4\theta^4 - 3\theta^5)}{24}.
\end{equation}
\break
\textbf{Note:} In equation 4.9, when the measures of dependency ($\theta$) are -1 and 1, the  Kendall's tau ($\tau(C)$) will be -1 and 1 respectively. \\

(iii) By the definition of Blomqvist's beta, we have
\begin{align*}
\beta(C) &= 4 C\left(\frac{1}{2},\frac{1}{2}\right)- 1.
\end{align*}
\begin{align*}
\textrm{Since} \quad C\left(\frac{1}{2},\frac{1}{2}\right) &= (1- \theta^2)\frac{1}{4} + \frac{\theta}{4}\left[(1+\theta)^2 \frac{1}{2} - 0\right] = (1- \theta^2)\frac{1}{4} + \frac{\theta(1+\theta)^2}{8},
\end{align*}
the Blomqvist's beta becomes
$$\beta(C) = 4\left[(1- \theta^2)\frac{1}{4} + \frac{\theta(1+\theta)^2}{8}\right]- 1 = \frac{\theta(\theta^2 + 1)}{2}\cdot$$

(iv) By the definition of Gini's gamma, we have
\begin{align*}
\gamma(C) &= 4\int_0^1[C(u,u) + C(u,1-u) - u]du, \quad \textrm{where} \\
C(u,u) &= (1-\theta^2)u^2 + \frac{\theta}{4}\left[(1+\theta)^2 u - (1-\theta)^2 \max(2u-1,0)\right], \quad \textrm{and}\\
\int_0^1 C(u,u)du &=  \int_0^1\left[(1-\theta^2)u^2 + \frac{\theta}{4}((1+\theta)^2 u - (1-\theta)^2 \max(2u-1,0))\right]du.
\end{align*}
After calculations, we get
\begin{align*}
\int_0^1 C(u,u)du &= (1-\theta ^2)\frac{1}{3}+ \frac{\theta}{4}\left[(1+\theta)^2 \frac{1}{2} - (1-\theta)^2 \frac{1}{4}\right] \\
&= \frac{3\theta^3 + 2\theta^2 + 3\theta +16}{48}.
\end{align*}
Next,\\
$$\int_0^1 C(u,1-u)du =  \int_0^1\left[(1-\theta^2)u(1-u) + \frac{\theta}{4}((1+\theta)^2 min(u,1-u) - (1-\theta)^2\cdot 0)\right]du $$
Since,$$\int_0^1(u - u^2)du = \frac{1}{6}, \int_0^1 min(u,1-u)du = \frac{1}{4},$$
$$\int_0^1 C(u,1-u)du = (1-\theta^2)\frac{1}{6} + \frac{\theta}{4}(1+ \theta)^2 \frac{1}{4}= \frac{3\theta^3-2\theta^2+3\theta + 8}{48},$$
we get,
\begin{align*}
\gamma(C) &= 4\left(\frac{3\theta^3 + 2\theta^2 + 3\theta +16}{48}+ \frac{3\theta^3-2\theta^2+3\theta + 8}{48}- \frac{1}{2}\right)\\
&= 4\left(\frac{6\theta^3 + 6\theta}{48}\right) = \frac{\theta(\theta^2 + 1)}{2}.
\end{align*}
\break
\textbf{Proof of theorem 2.3:}

The coefficient of the lower tail dependence is
$$ \lambda_L = \lim_{u \to 0}\frac{C(u,u)}{u}= \lim_{u \to 0} \frac{(1-\theta^2)u^2 + \frac{\theta}{4}(1+\theta)^2 u}{u}$$
$$= \lim_{u \to 0}\left[(1-\theta^2)u + \frac{\theta}{4}(1+\theta)^2\right]= \frac{\theta(1+\theta)^2}{4}.$$
And the coefficient of the upper tail dependence is
\begin{align*}
\lambda_U &= \lim_{u \to 1}\frac{1-2u + C(u,u)}{1-u} \\
&= \lim_{u \to 1}\left[\frac{1-2u + (1-\theta^2)u^2 + \frac{\theta}{4}((1+\theta)^2 u - (1-\theta)^2(2u-1))}{1-u}\right].
\end{align*}
Since this is of the $\frac{0}{0}$ form, we can use L'Hospital's rule to get,
\begin{align*}
\lambda_U &= \lim_{u \to 1} \left[\frac{-2 + 2(1-\theta^2)u + \frac{\theta}{4}((1+\theta)^2 - (1-\theta)^2\cdot 2)}{-1} \right] \\
&= -1\left[-2 + 2(1-\theta^2)+ \frac{\theta}{4}(1+2\theta+\theta^2-2+4\theta-2\theta^2)\right] \\
&= \frac{\theta^3 + 2\theta^2 + \theta}{4}= \frac{\theta(1+\theta)^2}{4}.
\end{align*}
\break
\textbf{Proof of proposition 2.1:}\\
\begin{description}
\item[Case(i)]  If $u \leq v$, $u + v \leq 1,$ \\
 $M(u,v)- W(u,v)$ = $\min(u,v) - \max(u+v-1, 0)$ \\
 $ = u - 0 = u.$
\item[Case(ii)] If $u \leq v$, $u + v > 1,$ \\
 $M(u,v)- W(u,v)$ = $\min(u,v) - \max(u+v-1, 0)$\\
 $ = u - (u+v-1) = 1 - v. $
\item[Case(iii)] If $u > v$, $u + v \leq 1,$ \\
 $M(u,v)- W(u,v)$ = $\min(u,v) - \max(u+v-1, 0)$ \\
 $ = v - 0 = v.$
\item[Case(iv)] If $u \leq v$, $u + v \leq 1,$ \\
 $M(u,v)- W(u,v)$ = $\min(u,v) - \max(u+v-1, 0)$\\
 $ = v - (u + v- 1)= 1 - u. $
\end{description}
\break
\textbf{Proof of theorem 2.4:}

To prove $C_{\alpha}(u,v)$ is a copula, we need to show that it satisfies all properties of a copula.
\begin{description}
\item[(i)] Boundary conditions: We have,
\begin{align*}
C_{\alpha}(u,0)&= C(u,0) +  \frac{\alpha}{2}\Big[\min(u,0) - \max(u-1,0)\Big]
 = 0 + \frac{\alpha}{2}(0-0) = 0, \\
  C_{\alpha}(0,v)&= C(0,v) +  \frac{\alpha}{2}\Big[\min(0,v) - \max(v-1,0)\Big]= 0 + \frac{\alpha}{2}(0-0) = 0, \\
  C_{\alpha}(1,v)&= C(1,v) +  \frac{\alpha}{2}\Big[\min(1,v) - \max(v,0)\Big] = v + \frac{\alpha}{2}(v-v) = v, \\
 C_{\alpha}(u,1)&= C(u,1) +  \frac{\alpha}{2}\Big[\min(u,1) - \max(u,0)\Big] = u + \frac{\alpha}{2}(u-u) = u.
 \end{align*}
 \item[(ii)]2-increasing property: For all $0\leq u_1 \leq u_2 \leq 1$,  $0 \leq v_1 \leq v_2 \leq 1$,
 \end{description}
 $ C_{\alpha}(u_1,v_1) + C_{\alpha}(u_2,v_2) - C_{\alpha}(u_1,v_2) + C_{\alpha}(u_2,v_1)$
\begin{align}
 &= C(u_1, v_1)+ \frac{\alpha}{2}\Big[\min(u_1,v_1)-\max(u_1+ v_1-1, 0)\Big] + C(u_2, v_2)+ \frac{\alpha}{2}\Big[\min(u_2,v_2)\nonumber \\
 &\quad -\max(u_2+ v_2-1, 0)\Big] - C(u_1, v_2)- \frac{\alpha}{2}\Big[\min(u_1,v_2)-\max(u_1+ v_2-1, 0)\Big] \nonumber \\
 &\quad - C(u_2, v_1)- \frac{\alpha}{2}\Big[\min(u_2,v_1)-\max(u_2+ v_1-1, 0)\Big]\nonumber  \\
 &= C(u_1, v_1) + C(u_2, v_2) - C(u_1, v_2) - C(u_2, v_1) +  \frac{\alpha}{2}\Big[\min(u_1,v_1)+ \min(u_2,v_2)\nonumber \\
 &\quad - \min(u_1,v_2)- \min(u_2,v_1)-\max(u_1+ v_1-1, 0)- \max(u_2+ v_2-1, 0) \nonumber \\
 &\quad +  \max(u_1+ v_2-1, 0)+ \max(u_2+ v_1-1, 0)\Big]
\end{align}
Since $C(u,v)$ is a copula, $C(u_1, v_1) + C(u_2, v_2) - C(u_1, v_2) - C(u_2, v_1) \geq 0,$ \\
 For the second part of expression (4.10), consider all four cases for $u$ and $v$,
 \begin{description}
 \item[(a)] For $ u \leq v$, $u + v \leq 1$, that is for  $ u_1 \leq v_1$,  $u_1 + v_1 \leq 1$ and $ u_2 \leq v_2$, $u_2 + v_2 \leq 1$, the second part of (3.4) will be \\
 $\frac{\alpha}{2}\Big[u_1 + u_2 -u_1 - \min(u_2, v_1)- 0 -0 + 0 + 0\Big]$\\
 $= \frac{\alpha}{2}\Big[u_2 - \min(u_2, v_1)\Big] \geq 0. $
\item[(b)] For $ u \leq v$, $u + v > 1$, that is for  $ u_1 \leq v_1$, $u_1 + v_1 > 1$ and $ u_2 \leq v_2$ and  $u_2 + v_2 > 1$, the second part of (3.4) will be $\frac{\alpha}{2}\Big[u_1 + u_2 -u_1 - \min(u_2, v_1)- (u_2 + v_2 - 1)- (u_1 + v_1 - 1)+ (u_1 + v_2 - 1)+(u_2 + v_1 - 1)\Big]$ \\
 $= \frac{\alpha}{2}\Big[u_2 -u_1 + u_2 - \min(u_2, v_1)\Big] \geq 0, $\\
 since $u_2 - u_1 \geq 0 $ and $u_2 - \min(u_2,v_1) \geq 0.$
\item[(c)] For $ u > v$, $u + v \leq 1,$ that is for  $ u_1 > v_1$, $u_1 + v_1 \leq 1$ and $u_2 > v_2$, $u_2 + v_2 \leq 1$, the second part of (3.4) will be $\frac{\alpha}{2}\Big[v_1 + v_2 - \min(u_1, v_2)- v_1 -0 -0 + 0 + 0\Big]$ \\
 $= \frac{\alpha}{2}\Big[v_2 - \min(u_1, v_2)\Big] \geq 0. $
\item[(d)] For $ u > v$, $u + v > 1,$ that is for  $ u_1 > v_1$, $u_1 + v_1 > 1$ and $ u_2 > v_2$, $u_2 + v_2 > 1$, the second part of (3.4) will be $\frac{\alpha}{2}\Big[v_1 + v_2 - \min(u_1, v_2)- v_1 - (u_1 + v_1 - 1)- (u_2 + v_2 - 1)+ (u_1 + v_2 - 1)+(u_2 + v_1 - 1)\Big]$ \\
 $= \frac{\alpha}{2}\Big[v_2 - \min(u_1, v_2)\Big] \geq 0. $
\end{description}
Hence, $C_{\alpha}(u,v)$ is 2-increasing. \\
So, $C_{\alpha}(u,v)= C(u,v) + \frac{\alpha}{2}\Big[\min(u,v) - \max(u+v-1, 0)\Big]$ is a copula.
\break
\textbf{Proof of theorem 2.5:}

By the definition of Spearman's rho,
\begin{align*}
\rho(C_{\alpha}) &= 12\int_0^1\int_0^1 C_{\alpha}(u,v)du dv - 3 \\
&= 12\int_0^1\int_0^1 \Big[C(u,v) + \frac{\alpha}{2}\Big(\text{M}(u,v) - \text{W}(u,v)\Big)\Big]du dv - 3 \\
&= 12\int_0^1\int_0^1 C(u,v)dudv - 3 + 12\int_0^1\int_0^1 \frac{\alpha}{2}\Big[\text{M}(u,v) - \text{W}(u,v)\Big]du dv \\
&= \rho(C) + 6\alpha \int_0^1\int_0^1 \Big[\text{M}(u,v) - \text{W}(u,v)\Big]du dv
\end{align*}
where
\begin{align*}
\text{M}(u,v) dudv &= \int_0^1\int_0^1 \min(u,v)\quad du dv \\
&= \int_0^1\int_0^1 u I_{u \leq v} dudv + \int_0^1\int_0^1 v I_{u > v}  dudv\\
&= \int_0^1\int_0^v u dudv + \int_0^1\int_v^1 v  dudv \\
&= \int_0^1\left[\frac{u^2}{2}\right]_{0}^{v} dv +  \int_0^1 v[u]_{v}^{1}dv
= \int_0^1 \frac{v^2}{2}dv +  \int_0^1 v(1-v) dv \\
&= \left[\frac{v^3}{6}\right]_{0}^{1} + \left[\frac{v^2}{2} - \frac{v^3}{3}\right]_{0}^{1} = \frac{1}{6} + \frac{1}{2} - \frac{1}{3} = \frac{1}{3}.
\end{align*}
and
\begin{align*}
\int_0^1\int_0^1 \text{W}(u,v)\quad dudv &= \int_0^1\int_0^1 \max(u+v-1, 0)\quad dudv \\
& = \int_0^1\int_0^1(u+v-1)I_{u \geq v-1} du dv = \int_0^1\int_{1-v}^1 (u+v-1) dudv \\
& = \int_0^1 \left[\frac{u^2}{2} + vu - u\right]_{1-v}^1 dv\\
& = \int_0^1\left[\frac{1}{2} + v -1 - \frac{(1-v)^2}{2} -v(1-v) + 1 - v\right]dv \\
& = \left[\frac{v}{2} + \frac{(v-1)^3}{6} - \frac{v^2}{2} + \frac{v^3}{3}\right]_{0}^{1} \\
&= \frac{1}{2}- \frac{1}{2} + \frac{1}{3}+ \frac{(-1)^3}{6}= \frac{1}{6}.
\end{align*}
Hence, \\
$\rho(C_{\alpha})= \rho(C) + 6\alpha(\frac{1}{3} - \frac{1}{6})= \rho(C) + \alpha $ where $\alpha \in [-1,1]. $ \\

By the definition of Blomqist's beta,
$$ \beta(C_{\alpha}) = 4 C_{\alpha}\left(\frac{1}{2},\frac{1}{2}\right)- 1.$$

\begin{align*}
&= 4\left[C\left(\frac{1}{2},\frac{1}{2}\right) + \frac{\alpha}{2}\cdot \min \left(\frac{1}{2}, \frac{1}{2}\right)- \frac{\alpha}{2}\cdot\max \left(\frac{1}{2} + \frac{1}{2}- 1, 0\right)\right] \\
& = 4C \left(\frac{1}{2},\frac{1}{2}\right) + 4.\frac{\alpha}{2}\left(\frac{1}{2} - 0\right) -1 \\
& = 4C\left(\frac{1}{2},\frac{1}{2}\right)- 1 + \alpha \\
&= \beta(C) + \alpha.
\end{align*}

By the definition of Gini's gamma, we have
\begin{align*}
\gamma(C_{\alpha}) &= 4\int_0^1\Big[C_{\alpha}(u,u) + C_{\alpha}(u, 1-u) - u\Big]du, \quad \textrm{where} \\
C_{\alpha}(u,u)&= C(u,u) + \frac{\alpha}{2}\Big[\min(u,u) - \max(u+u-1,0)\Big] \\
&= C(u,u) + \frac{\alpha}{2}\Big[u - (2u - 1)I_{u \geq \frac{1}{2}}\Big], \quad \textrm{and}, \\
\int_0^1 C_{\alpha}(u,u)du &=  \int_0^1 C(u,u)du + \frac{\alpha}{2} \int_0^1 \Big[u - (2u - 1)I_{u \geq \frac{1}{2}}\Big]du \\
& = \int_0^1 C(u,u)du + \frac{\alpha}{2}\Big(\frac{1}{2}- \frac{1}{4}\Big)= \int_0^1 C(u,u)du + \frac{\alpha}{8}.
\end{align*}
and
\begin{align*}
\int_0^1 C_{\alpha}(u,1-u)du &=  \int_0^1 C(u,1-u)du + \frac{\alpha}{2} \int_0^1\Big[\min(u,1-u) - \max(u+1\\
&\quad -u-1, 0)\Big]du \\
&= \int_0^1 C(u,1-u)du + \frac{\alpha}{2}\Big[\int_0^{\frac{1}{2}} u du + \int_{\frac{1}{2}}^{1}(1-u)du\Big]\\
& = \int_0^1 C(u,1-u)du + \frac{\alpha}{2}\left(\frac{1}{8} + \frac{1}{8}\right) = \int_0^1 C(u,1-u)du + \frac{\alpha}{8}.
\end{align*}
Hence,
\begin{align*}
\gamma(C_{\alpha})&= 4\Big[\int_0^1 C(u,u)du + \int_0^1 C(u,1-u)du + \frac{\alpha}{8} - \int_0^1 udu\Big] \\
&= 4\int_0^1\Big[C(u,u) + C(u,1-u) -u \Big]du + 4\Big(\frac{\alpha}{8}+ \frac{\alpha}{8}\Big)\\
&= \gamma(C) + \alpha.
\end{align*}

By the definition of Kendall's tau, we have
\begin{align*}
\tau(C_{\alpha})&= 4\int_0^1\int_0^1 C_{\alpha}(u,v)\frac{\partial ^2 C_{\alpha}}{\partial u \partial v} dudv - 1 \\
&= 4\int_0^1\int_0^1 \left[C(u,v)+ \frac{\alpha}{2}\Big(M(u,v) - W(u,v)\Big)\right]\frac{\partial ^2} {\partial u \partial v}\Big[C(u,v) \\
&\quad + \frac{\alpha}{2}\Big(M(u,v) - W(u,v)\Big)\Big]dudv - 1
\end{align*}
Since
\begin{equation}
  M(u,v)- W(u,v) =
    \begin{cases}
      u, & \text{if $u \leq v$, $u + v \leq 1$}\\
      1-v, & \text{if $u \leq v$, $u + v > 1$}\\
      v,  &\text{if $u > v$, $u + v \leq 1$}\\
      1-u, & \text{if $u > v$, $u + v > 1$},\\
    \end{cases}
\end{equation}
$$ \frac{\partial ^2} {\partial u \partial v}\Big[M(u,v) - W(u,v)\Big]= 0. $$
We can get,
\begin{align*}
\tau(C_{\alpha})&= 4\int_0^1\int_0^1\Big[C(u,v)+ \frac{\alpha}{2}\Big(M(u,v) - W(u,v)\Big)\Big]\frac{\partial ^2 C} {\partial u \partial v} dudv - 1 \\
&=4\int_0^1\int_0^1 C \frac{\partial ^2 C} {\partial u \partial v} dudv - 1 + 4\int_0^1\int_0^1 \frac{\alpha}{2}\Big[M(u,v) - W(u,v)\Big]\frac{\partial ^2 C} {\partial u \partial v} dudv \\
& = \tau(C)+ 2\alpha \int_0^1\int_0^1\Big[M(u,v)- W(u,v)\Big]\frac{\partial ^2 C}{\partial u \partial v} dudv. \\
\end{align*}

\textbf{Proof of theorem 2.6:}\\

We have,
\begin{align*}
\lambda_{L}^{C_{\alpha}}&= \lim_{u \to 0}\frac{C_{\alpha}(u,u)}{u} \\
&=\lim_{u \to 0}\left[\frac{C(u,u) +  \frac{\alpha}{2}\Big(\min(u,u) - \max(u+u-1,0)\Big)}{u}\right] \\
&=\lim_{u \to 0}\left[ \frac{C(u,u) + \frac{\alpha}{2}(u-0)}{u}\right]= \lim_{u \to 0} \frac{C(u,u)}{u} + \lim_{u \to 0} \frac{\alpha}{2}\\
&= \lambda_L^C + \frac{\alpha}{2}.\\
\end{align*}
and
\begin{align*}
\lambda_{U}^{C_{\alpha}}&= \left[\lim_{u \to 1} \frac{1 - 2u + C_{\alpha}(u,u)}{1-u}\right]\\
&= \lim_{u \to 1}\left[\frac{1 - 2u + C(u,u)+ \frac{\alpha}{2}(\min(u,u)-\max(u+u-1, 0))}{1-u}\right] \\
&= \lim_{u \to 1}\left[ \frac{1 - 2u + C(u,u)}{1-u}\right] + \lim_{u \to 1}\left[\frac{\alpha(u - (2u - 1))}{2(1-u)}\right] \\
& = \lambda_{U}^{C} + \lim_{u \to 1} \frac{\alpha}{2}= \lambda_{U}^{C} + \frac{\alpha}{2}.
\end{align*}
\bibliographystyle{ieeetran}
\bibliography{ref}

\end{document}